\documentclass[reqno,11pt]{article}
\usepackage{bbm}
\usepackage{amsfonts}
\usepackage{amsmath,amssymb,amsthm}
\usepackage{color}
\usepackage[colorlinks]{hyperref}
\usepackage{mathrsfs}
\usepackage{epstopdf}
\usepackage{epsfig}
\usepackage{graphicx}
 \theoremstyle{plain}

\newtheorem{theorem}{Theorem}[section]
\newtheorem{lemma}{Lemma}[section]

\theoremstyle{remark}

\newtheorem{example}{Example}[section]
\theoremstyle{definition}
\newtheorem{definition}{Definition}[section]

\numberwithin{equation}{section}

\newcommand{\ba}{\begin{array} }
\newcommand{\ea}{\end{array} }
\newcommand{\be}{\begin{equation} }
\newcommand{\ee}{\end{equation} }

\newcommand{\f}{\displaystyle\frac}

\begin{document}
\title{\bf Periodic Solutions for SDEs through Upper and Lower Solutions
\author{ Chunyan Ji$^{a,b}$\,\,\,\,\,\,Xue Yang$^{a,c}$ \,\,\,\,\,\, Yong
Li$^{c,d,a}$\thanks{Corresponding author at: School of Mathematics, 
Jilin University, Changchun, 130012, P.R.
China.}}
\thanks{ E-mail
addresses: chunyanji80@hotmail.com(C. Ji),
yangxuemath@163.com(X.Yang), liyongmath@163.com(Y.Li).}}

\date{}
\maketitle \setcounter{page}{1} \pagestyle{plain}
\begin{center}{{\scriptsize $^a$ School of Mathematics and Statistics \& Center for
Mathematics and Interdisciplinary Sciences, Northeast Normal
University, Changchun, 130024, P.R. China.}\\
{\scriptsize $^b$ School of Mathematics and Statistics, Changshu
Institute of Technology, Changshu 215500,  P.R. China.}\\
{\scriptsize $^c$ School of Mathematics, Jilin University,
Changchun,
130012, P.R. China}\\
{\scriptsize $^d$ State Key Laboratory of Automotive Simulation and
Control, Jilin University, 130025, P.R. China}}
\end{center}

\noindent {\bf Abstract:} We study a kind of better recurrence than
Kolmogorov's one: periodicity recurrence, which corresponds periodic
solutions in distribution for stochastic differential equations. On
the basis of technique of upper and lower solutions and comparison
principle, we obtain the existence of periodic solutions in
distribution for stochastic differential equations (SDEs). Hence
this provides an effective method how to study the periodicity of
stochastic systems by analyzing deterministic ones. We also
illustrate our results.


\noindent {\bf 2000 MR Subject Classification.}

\section{Introduction}
This paper concerns a kind of better recurrence: periodicity recurrence, that is, periodic solutions in
distribution of the following stochastic differential equation (SDE
for short):
\begin{equation}\label{SDE}
dX(t)=f(t,X(t))dt+g(t,X(t))dB(t).
\end{equation}
So far this has been yet paid rare attention relative to the
existence of stationary solutions. It is well known that the
existence problem of periodic solutions is one of center topics in
the qualitative theory of deterministic differential equations for
its significance in the physical science\cite{Krasnoselskii1968}.
There has been a large amount of work (see for example
\cite{CandidoLlibreNovaes2017,Fokam2017,TaralloZhou2018} and the
references therein). However, for SDEs, the existence of periodic
solutions is thought to be a challenging problem. Certainly,
periodic solutions of a deterministic system are not always
persistent under diffusion. Naturally, one asks when this better
periodicity persists for stochastic systems; more precisely, when the
ordinary differential equation (ODE for short)
$$
dX=f(t,X)dt
$$
has periodic solutions, does SDE \eqref{SDE} still admit periodic solutions in distribution?

In the present paper, we will
touch the problem. We find that the answer will be affirmative if the ODE has upper and lower
solutions. This conclusion is somewhat unexpected, because no additional condition is added to the diffusion term besides the usual one.  We believe that it is best possible to pose periodic solutions in distribution. As one pointed out  ( see, for example, \cite{Mellah2013}), it is impossible to obtain periodic solutions in probability or moment for SDEs  due to the effects of diffusion.

Let us recall that there are many topological and analytic methods,
such as  degree theory, fixed point theorems  in studying the
existence of periodic solutions of deterministic differential
equations. But for SDEs, these nonlinear methods do not work in
general, due to lack of compactness. Khasminskii
\cite{Khasminskii2012} defined periodic solutions in the sense of
periodic Markov process. Recently, Ji et al. \cite{JiYi2019} studied periodic probability solutions to be periodic analogs of
stationary measures for stationary Fokker-Planck equations. Chen et
al. \cite{ChenHanLiYang2017} gave a criterion analogous to Halanay's
criterion to prove the existence of periodic solutions in
distribution. Liu and Sun \cite{LiuSun2014} established the
existence of solutions which are almost automorphic in distribution
for some semilinear SDEs with  L\'{e}vy noise. Liu and Wang
\cite{LiuWang2016} obtained  almost periodicity in distribution by
Favard separation method. Tudor \cite{Tudor1992} proved the almost
periodicity of the one-dimensional distributions of solutions under
some hypotheses. Prato and Tudor \cite{Prato1995} showed the
existence of periodic and almost periodic solutions in distribution
of semilinear stochastic equations on a separable Hilbert space. Zhao and Zheng \cite{Zhao2009} and 
Feng and Zhao \cite{FengZhaoZhou2011,FengWuZhao2016} made some interesting 
investigation on periodic solutions for SDEs in a kind of local periodicity.

Additionally, there are a lot literature about  monotone methods and comparison
arguments in deterministic dynamical systems (see, for example, \cite{Smith1996}).
Especially, the upper and lower solutions method is an effective tool
in dealing with periodic solutions
\cite{KaulVatsala1986,LakshmtkanthamLeela1984,LiWangLiLu1999,LiCongLinLiu1999,SedaNietoCera1992}.
In this paper, on the basis of upper and lower solutions, using stochastic
comparison technique we try to prove the existence of periodic
solutions in distribution. Of course comparison principle
\cite{BaiJiang2017,IkedaWatanabe1977,PengZhu2006} is also powerful
to investigate dynamics of nonlinear systems. However, to our
knowledge, it seems the first time to give periodic solutions
in distribution for SDEs by combining these technique. Therefore this
paper provides some way to tackle the existence of periodic
solutions in distribution.

This paper is organized as follows. In Section 2, we review some
concepts and introduce some notations. In Sections 3 and 4, we show
the existence of periodic solutions in distribution of scalar SDEs
and multi-dimensional SDEs, respectively, via a stochastic comparison
approach. In the last section, some examples are given to
illustrate the theoretical results.

\section{Preliminary}

Throughout the paper, let $(\Omega, \{\mathcal{F}_t\}_{t\geq 0}, P)$
be a complete probability space with a filtration
$\{\mathcal{F}_t\}_{t\geq 0}$ satisfying the usual conditions (i.e.
it is right continuous and $\mathcal{F}_0$ contains all $P$-null
sets). $L^2(P,\mathbb{R}^n)$ stands for the space of all
$\mathbb{R}^n$-valued random variables $X$ such that
$E|X|^2=\int_{\Omega}|X|^2dP<\infty$. For $X\in
L^2(P,\mathbb{R}^n)$, let
$\|X\|_2:=\left(\int_{\Omega}|X|^2dP\right)^{1/2}$. Then
$L^2(P,\mathbb{R}^n)$ is a Hilbert space equipped with the norm
$\|\cdot\|_2$.  For an $\mathbb{R}^n$-valued random process
$X=\{X(t):t\in[0,K]\}$, if $\sup_{t\in[0,K]}\|X(t)\|_2<\infty$, then
$X$ is $L^2$-bounded, where $K$ is a positive constant. Then the set
of $L^2$-bounded stochastic processes is a Banach space. Let
$L^2_{\mathcal{F}_t}([0,K];\mathbb{R}^n)$ denote the family of all
${\mathcal{F}_t}$-measurable $C^1([0,K];\mathbb{R}^n)$-valued random
variables $X$ such that $\sup_{t\in[0,K]}\|X\|_2<\infty$. If
$X(k_1)$ and $X(k_2)$ are equal in distribution, we denote it by
$X(k_1)\overset{d}{=}X(k_2)$, where $k_1,k_2$ are two constants. For
two vectors $x=(x_1,x_2,\cdots,x_l), y=(y_1,y_2,\cdots,y_l)$, we say
$x<y$ (or $x\leq y$) if $x_i<y_i$ (or $x_i\leq y_i$),
$i=1,2,\cdots,l$. $a\wedge b$ denotes $\min\{a,b\}$.

Consider the system
\begin{equation}\label{bvp}
x'=h(t,x),
\end{equation}
where $h:[0,\theta]\times\mathbb{R}^l\rightarrow \mathbb{R}^l$ is a
continuous function.

We recall the conception about upper and lower solutions for \eqref{bvp}:

\begin{definition}\label{uls} \cite{BernfeldLakshmtkantham1974}
$C^1$-functions $\alpha, \beta: [0,\theta]\rightarrow\mathbb{R}^l$
are said to be a strictly lower solution and  a strictly upper
solution of system \eqref{bvp}, respectively,
 if $\alpha(t)< \beta(t)$ for $t\in[0, \theta]$ and
\begin{align*}
\left\{\begin{array}{c}
\alpha'< h(t,\alpha),\\
\alpha(0)\leq \alpha(\theta);
\end{array}
\right.~~~~~~~~~~~~~~~~ \left\{\begin{array}{c}
 \beta'> h(t,\beta),\\
\beta(0)\geq\beta(\theta).
\end{array}
\right.
\end{align*}
\end{definition}

For settings bolow, we give  the following function \cite{BuckdahnPend1999}:
\begin{equation}\label{vy}
\varphi_{\epsilon}(y)=\left\{
\begin{array}{ll}
y^2,&y\leq 0,\\
y^2-\f{y^3}{6\epsilon},&0<y\leq 3\epsilon,\\
2\epsilon y-\f{4}{3}\epsilon^2,&y>2\epsilon.
\end{array}\right.
\end{equation}
It is easy to see that $\varphi_{\epsilon}(y)\in C^2(\mathbb{R}),
\varphi_{\epsilon}'(y)\rightarrow 2y^{-}$ uniformly with respect to
$y$, $\varphi_{\epsilon}''(y)\rightarrow 2I_{y\leq 0}$ and
$\varphi_{\epsilon}(y)\rightarrow |y^{-}|^2$ provided that
$\epsilon\rightarrow 0$, where $y^{-}=y\wedge 0$.

\section{Scalar SDEs}
Consider the following scalar SDE
\begin{equation}\label{sde}
dx(t)=f(t,x(t))dt+g(t,x(t))dB(t),
\end{equation}
where $B(t)$ is a one dimensional Gaussian process with values in
$\mathbb{R}$ which is ${\cal{F}}_t$-adapted. Assume the drift term
 and the
diffusion term
$f,g:\mathbb{R}_+\times\mathbb{R}\rightarrow\mathbb{R}$ are
continuous and satisfy
\begin{align*}
{\bf{H}}:~~~~~
&f(t+\theta,x)=f(t,x), g(t+\theta,x)=g(t,x),\\
&|f(t,x)-f(t,y)|\leq M|x-y|, t\in[0,\theta]~~\mbox{for any}~~x,y\in[\alpha,\beta],\\
&|g(t,x)-g(t,y)|^2\leq L|x-y|^2, t\in[0,\theta]~~\mbox{for
any}~~x,y\in[\alpha,\beta],
\end{align*}
where  $\theta, M$ and $L$ are positive constants,
$\alpha(t),\beta(t)$ are strictly lower and upper solutions of
system $x'=f(t,x)$ defined by Definition \ref{uls}.

We need a stochastic version of comparable principles, which is a key for our arguments. For this, define two SDEs by
\begin{equation}\label{alpha}
\begin{split}
d\tilde{\alpha}(t)&=[\alpha'(t)-M(\tilde{\alpha}(t)-\alpha(t))]dt+g(t,\tilde{\alpha}(t))dB(t)\\
&:=f_1(t,\tilde{\alpha}(t))dt+g(t,\tilde{\alpha}(t))dB(t),
\end{split}
\end{equation}
\begin{equation}\label{beta}
\begin{split}
 d\tilde{\beta}(t)&=[\beta'(t)-M(\tilde{\beta}(t)-\beta(t))]dt+g(t,\tilde{\beta}(t))dB(t)\\
 &:=f_2(t,\tilde{\beta}(t))dt+g(t,\tilde{\beta}(t))dB(t)
 \end{split}
\end{equation}
with initial values $\tilde{\alpha}(0)=\alpha(0)+\xi$ and
$\tilde{\beta}(0)=\beta(0)-\xi$, respectively and $t\in[0,\theta]$,
where $\alpha(t),\beta(t)$ are defined by Definition \ref{uls} with
$h(t,x)=f(t,x)$, $\xi>0$ is a sufficiently small constant. Then
there are solutions $\tilde{\alpha}(t), \tilde{\beta}(t)$ for
$t\in[0,\theta]$, which have  the following
 property.

\begin{lemma}\label{abt}
Let $\tilde{\alpha}(t), \tilde{\beta}(t)~~(t\in[0,\theta])$ be
solutions of equations \eqref{alpha} and \eqref{beta} with initial
values $\tilde{\alpha}(0)$ and $\tilde{\beta}(0)$, respectively.
Assume the second condition in H holds. Then
\begin{align*}
\alpha(t)\leq
\tilde{\alpha}(t)\leq\tilde{\beta}(t)\leq\beta(t),~t\in[0,\theta]~a.s.
\end{align*}
\end{lemma}

\begin{proof}
We only need to prove $\tilde{\alpha}(t)\geq\alpha(t)$~a.s. The
proof of $\tilde{\beta}(t)\leq\beta(t)$~~a.s. is similar. Since
${\alpha}(0)<\tilde{\alpha}(0)$, then $
{\alpha}(t)<\tilde{\alpha}(t)$ for $t\in[0,\tau)$, where
\begin{align*}
\tau=\inf\{t\geq 0: {\alpha}(t)\geq\tilde{\alpha}(t)\}.
\end{align*}
We need to show that $\tau>\theta$. Let $m_0$ be a positive integer
such that $\tilde{\alpha}(0)-{\alpha}(0)\geq \f{1}{m_0}$. For each
integer $m\geq m_0$, define the stopping time
\begin{align*}
\tau_m=\inf\{t\in[0,\tau): \tilde{\alpha}(t)-{\alpha}(t)\leq
{1}/{m}\}.
\end{align*}
Clearly, $\tau_m$ is increasing as $m\rightarrow\infty$. Set
$\tau_{\infty}=\lim\limits_{m\rightarrow\infty}\tau_m$. Hence
$\tau_{\infty}\leq \tau$ a.s. If we can show that
$\tau_{\infty}>\theta$ a.s.,  then $\tau>\theta$ a.s. and
${\alpha}(t)<\tilde{\alpha}(t)$ for $t\in[0,\theta]$ a.s. Suppose
this statement is not true, then there is  a pair of constant
$0<\gamma\leq\theta$ and $0<\zeta<1$ such that
\begin{align*}
P\{\tau_\infty\leq\gamma\}>\zeta.
\end{align*}
Hence there is an integer $m_1\geq m_0$ such that
\begin{align*}
P\{\tau_m\leq\gamma\}\geq\zeta~~\mbox{for all}~~m\geq m_1.
\end{align*}
Integrating equation \eqref{alpha} from $0$ to $\tau_m\wedge \gamma$
yields
\begin{align*}
&(\tilde{\alpha}(\tau_m\wedge\gamma)-{\alpha}(\tau_m\wedge\gamma))-(\tilde{\alpha}(0)-{\alpha}(0))\\
&=\int_0^{\tau_m\wedge\gamma}-M(\tilde{\alpha}(t)-{\alpha}(t))dt+\int_0^{\tau_m\wedge\gamma}g(t,\tilde{\alpha}(t))dB(t).
\end{align*}
Then
\begin{align*}
e^{M(\tau_m\wedge\gamma)}(\tilde{\alpha}(\tau_m\wedge\gamma)-{\alpha}(\tau_m\wedge\gamma))-(\tilde{\alpha}(0)-{\alpha}(0))
=\int_0^{\tau_m\wedge\gamma}e^{Mt}g(t,\tilde{\alpha}(t))dB(t).
\end{align*}
Multiplying $I_{\{\tau_m\leq \gamma\}}$ on the both side of it and
taking expectation yield
\begin{align*}
&E[I_{\{\tau_m\leq
\gamma\}}(\tilde{\alpha}(\tau_m\wedge\gamma)-{\alpha}(\tau_m\wedge\gamma))]=
E[ I_{\{\tau_m\leq
\gamma\}}(\tilde{\alpha}(0)-{\alpha}(0))e^{-M\tau_m}].
\end{align*}
 Note that the left side
\begin{align*}
E[I_{\{\tau_m\leq
\gamma\}}(\tilde{\alpha}(\tau_m\wedge\gamma)-{\alpha}(\tau_m\wedge\gamma))]\leq\f{1}{m},
\end{align*}
while the right side
\begin{align*}
E[ I_{\{\tau_m\leq
\gamma\}}(\tilde{\alpha}(0)-{\alpha}(0))e^{-M\tau_m}] \geq E[
I_{\{\tau_m\leq
\gamma\}}(\tilde{\alpha}(0)-{\alpha}(0))]e^{-M\gamma} \geq\zeta\xi
e^{-M\gamma}.
\end{align*}
Letting $m\rightarrow\infty$ leads to a contradiction that $0\geq
\zeta\xi e^{-M\gamma}>0$. Therefore we must have
$\tau_{\infty}>\theta$ a.s.

As to the result
$\tilde{\alpha}(t)\leq\tilde{\beta}(t),~t\in[0,\theta]$ a.s., it is
obviously true through comparison theorem for stochastic
differential equations under the second condition in H. This
completes the proof.
\end{proof}

The following is the first main result about periodic solutions:
\begin{theorem}\label{1th}
Let the hypothesis $H$ hold, where $M-L>\f{\ln 2}{2\theta}$ is
satisfied, and assume there are strictly lower and upper solutions
$\alpha$ and $\beta$ of system $x'=f(t,x)$ with $\alpha<\beta$. Then
there exist monotone sequences $\{\tilde{\alpha}_n(t)\},
\{\tilde{\beta}_n(t)\}$ with $\alpha_0=\tilde{\alpha},
\beta_0=\tilde{\beta}$ and functions $a(t), b(t)$ such that
$\lim\limits_{n\rightarrow\infty}\tilde{\alpha}_n(t)=a(t),
\lim\limits_{n\rightarrow\infty}\tilde{\beta}_n(t)=b(t)$ and
\begin{align*}
\tilde{\alpha}=\tilde{\alpha}_0\leq\tilde{\alpha}_1\leq\cdots\leq\tilde{\alpha}_n\leq
a\leq u\leq
b\leq\tilde{\beta}_n\leq\cdots\leq\tilde{\beta}_1\leq\tilde{\beta}_0=\tilde{\beta}
\end{align*}
on $[0,\theta]$, where $u$ is a solution of system \eqref{sde} such
that $\tilde{\alpha}(t)\leq u(t)\leq\tilde{\beta}(t)$ on
$[0,\theta]$ a.s., and $u(0)\overset{d}{=}u(\theta)$. Therefore, there is a $\theta-$periodic solution $x^*(t)$ in
distribution of system \eqref{sde}.
\end{theorem}

\begin{proof}
To prove the first result of this theorem, we divide it into three
steps.
\begin{itemize}
\item Step 1:  Construct an auxiliary equation, and prove that it
has a unique solution.
\item Step 2: From Step 1, we define a mapping, and show it has two order 
properties.
\item Step 3: From Step 2, we can find monotone sequences, and so
the result holds.
\end{itemize}

{\bf Step 1:}
 Consider the
following equation
\begin{equation}\label{FZsde}
du=(f(t,\eta)-M(u-\eta))dt+g(t,u)dB(t),
\end{equation}
where $\eta\in[\tilde{\alpha},\tilde{\beta}]=\{u\in
L^2([0,\theta];\mathbb{R}),\tilde{\alpha}\leq
u\leq\tilde{\beta}~~\mbox{a.s.}\}$.
 It is easy to see that
\begin{align*}
u(t)=u(0)e^{-Mt}+\int_0^te^{-M(t-s)}(f(s,{\eta}(s))+M{\eta}(s))ds+\int_0^te^{-M(t-s)}g(s,u(s))d{B}(s)\triangleq
Tu,
\end{align*}
where $T: \{u\in L^2([0,\theta];\mathbb{R}),\tilde{\alpha}\leq
u\leq\tilde{\beta}~~\mbox{a.s.}\}\rightarrow \{u\in
L^2([0,\theta];\mathbb{R}),\tilde{\alpha}\leq
u\leq\tilde{\beta}~~\mbox{a.s.}\}$. Obviously, there is a unique
solution $u(t), t\in[0,\theta]$ of system \eqref{FZsde} from the
continuity of its coefficients. That is to say, for each
$\eta\in[\tilde{\alpha},\tilde{\beta}]$ and the initial value
$\gamma\in L^2([0,\theta];\mathbb{R}), \tilde{\alpha}(0)\leq
\gamma\leq\tilde{\beta}(0)$, there exists a solution
$\tilde{\alpha}(t)\leq u_{\gamma,\eta}(t)\leq\tilde{\beta}(t)$.
Moreover, we can claim that for each fixed
$\eta\in[\tilde{\alpha},\tilde{\beta}]$, there exists a unique
solution in distribution of \eqref{FZsde} with
$u(0)\overset{d}{=}u(\theta)$. In fact,
\begin{align*}
E|u(t)-v(t)|^2&=E\left|(u(0)-v(0))e^{-M t}+\int_0^t e^{-M(t-s)}(g(s,u(s))-g(s,v(s)))dB(s)\right|^2\\
&\leq
2E\left|(u(0)-v(0))e^{-M t}\right|^2+2E\int_0^t e^{-2M(t-s)}|g(s,u(s))-g(s,v(s))|^2ds\\
&\leq 2e^{-2M t}E\left|(u(0)-v(0))\right|^2+2L\int_0^t
e^{-2M(t-s)}E|u(s)-v(s)|^2ds.
\end{align*}
This together with Gronwall's inequality implies that
\begin{align*}
E|u(t)-v(t)|^2\leq 2e^{-2(M-L)t}E\left|(u(0)-v(0))\right|^2.
\end{align*}
Then
\begin{align*}
E|u(\theta)-v(\theta)|^2\leq
2e^{-2(M-L)\theta}E\left|(u(0)-v(0))\right|^2.
\end{align*}
Define $\tilde{T}u(0)=Tu(\theta)$. Noting that
$\tilde{\alpha}(0)\leq\tilde{\alpha}(\theta),
\tilde{\beta}(0)\geq\tilde{\beta}(\theta)$ and by Lemma \ref{abt} we
have $\tilde{T}: {[\tilde{\alpha}(0),~\tilde{\beta}(0)]}\rightarrow
{[\tilde{\alpha}(0),~\tilde{\beta}(0)]}$. Together with
$2e^{-2(M-L)\theta}<1$, i.e. $M-L>\f{\ln 2}{2\theta}$, we know that 
$\tilde{T}$ is contract. Therefore, according to Banach's
contraction principle, there is a $\gamma^*$, such that
$\tilde{\alpha}(t)\leq u_{\gamma^*,\eta}^*(t)\leq\tilde{\beta}(t)$
and
$u_{\gamma^*,\eta}^*(0)\overset{d}{=}u_{\gamma^*,\eta}^*(\theta)$ 
of system \eqref{FZsde}.

{\bf Step 2:} Now define a mapping $A(\eta)=u$, where $u$ is the
unique solution of \eqref{FZsde} with $u(0)\overset{d}{=}u(\theta)$.
The mapping $A$ is continuous. In fact, for any
$\eta_1,\eta_2\in[\tilde{\alpha},\tilde{\beta}]$ and
$\|\eta_1-\eta_2\|_2<\epsilon$, then
\begin{align*}
&E|A(\eta_1)-A(\eta_2)|^2\\
&=E\left|\int_0^te^{-M(t-s)}(f(s,{\eta_1}(s)-f(s,{\eta_2}(s))+M({\eta_1}(s)-{\eta_2}(s))ds+\int_0^te^{-M(t-s)}(g(s,u_1(s))-g(s,u_2(s)))d{B}(s)\right|^2\\
&\leq
3E\left|\int_0^te^{-M(t-s)}(f(s,{\eta_1}(s)-f(s,{\eta_2}(s))ds\right|^2+3E\left|\int_0^te^{-M(t-s)}M({\eta_1}(s)-{\eta_2}(s))ds\right|^2\\
&\quad\quad+3E\left|\int_0^te^{-M(t-s)}(g(s,u_1(s))-g(s,u_2(s)))d{B}(s)\right|^2\\
&\leq 3\int_0^te^{-2M(t-s)}ds E\int_0^t|f(s,{\eta_1}(s))-f(s,{\eta_2}(s))|^2ds+3\int_0^te^{-2M(t-s)}ds E\int_0^tM^2|{\eta_1}(s)-{\eta_2}(s)|^2ds\\
&\quad\quad+3E\int_0^te^{-2M(t-s)}|g(s,u_1(s))-g(s,u_2(s))|^2ds\\
&\leq3ME\int_0^t|{\eta_1}(s)-{\eta_2}(s)|^2ds+3L\int_0^tE|A(\eta_1)-A(\eta_2)|^2ds.
\end{align*}
It is easy to see the result is true through Gronwall's inequality.
Moreover $A$ satisfies
\begin{itemize}
\item[](i)~$\tilde{\alpha}\leq A(\tilde{\alpha}), \tilde{\beta}\geq
A(\tilde{\beta})$~~a.s.;
\item[](ii)~for $\eta_1,\eta_2\in[\tilde{\alpha},\tilde{\beta}]$,
 $\eta_1\leq \eta_2$~~a.s. implies $A\eta_1\leq A\eta_2$ a.s., i.e., the mapping $A$ possesses a monotone property on the segment
$[\tilde{\alpha}, \tilde{\beta}]$.
\end{itemize}

To prove (i), let $\varepsilon_n>0$ be a strictly decreasing
sequence
 with $\lim\limits_{n\rightarrow\infty}\varepsilon_n=0$. Define
 $f_1^{\varepsilon_n}(t,x)=f_1(t,x)-\varepsilon_n$ and
\begin{align*}
d\tilde{\alpha}_{n0}(t)&=f_1^{\varepsilon_n}(t,\tilde{\alpha}_{n0}(t))dt+g(t,\tilde{\alpha}_{n0}(t))dB(t)\\
&=(\alpha'(t)-M(\tilde{\alpha}_{n0}(t)-\alpha(t))-\varepsilon_n)dt+g(t,\tilde{\alpha}_{n0}(t))dB(t)
 \end{align*}
with the same initial value as the initial value of
$\tilde{\alpha}(t)$. Obviously,
\begin{equation}\label{fn}
 f_1^{\varepsilon_1}(t,x)<f_1^{\varepsilon_2}(t,x)<\cdots<f_1^{\varepsilon_n}(t,x)<f_1(t,x).
\end{equation}
Comparison theorems and \eqref{fn} imply that
\begin{align*}
\tilde{\alpha}_{10}(t)\leq\tilde{\alpha}_{20}(t)\leq\cdots\tilde{\alpha}_{n0}(t)\leq\tilde{\alpha}(t)~~\mbox{a.s.}~~\mbox{for
} ~~t\in[0,\theta].
\end{align*}
We can show that
\begin{equation}\label{limitalphaP}
\lim\limits_{n\rightarrow\infty}\tilde{\alpha}_{n0}(t)=\tilde{\alpha}(t)~~\mbox{a.s.}~~\mbox{for
}~~t\in[0,\theta].
\end{equation}
For this, let
\begin{equation}\label{limitalpha}
\tilde{\alpha}^0(t)\triangleq\lim\limits_{n\rightarrow\infty}\tilde{\alpha}_{n0}(t)~~\mbox{a.s.}~~\mbox{for
}~~t\in[0,\theta].
\end{equation}
We need to show that $\tilde{\alpha}^0(t)=\tilde{\alpha}(t)$ a.s.
for $t\in[0,\theta]$, which reduce to check that
$\tilde{\alpha}^0(t)$ satisfies \eqref{alpha} according to the
uniqueness of the strong solution.

For this aim, we first prove that $\tilde{\alpha}_{n0}(t)$ converges
to $\tilde{\alpha}^0(t)$ uniformly in $t\in[0,\theta]$ a.s. as
$n\rightarrow\infty$. Note that
\begin{align*}
&~~~~\sup\limits_{0\leq s\leq
t}|\tilde{\alpha}_{(n+1)0}(s)-\tilde{\alpha}_{n0}(s)|^2\\
&=\sup\limits_{0\leq s\leq t}\left|
\int_0^s(f_1^{\varepsilon_{n+1}}(u,\tilde{\alpha}_{(n+1)0}(u))-f_1^{\varepsilon_n}(u,\tilde{\alpha}_{n0}(u)))du
+\int_0^s(g(u,\tilde{\alpha}_{(n+1)0}(u))-g(u,\tilde{\alpha}_{n0}(u)))dB(u)\right|^2\\
&=\sup\limits_{0\leq s\leq
t}\left|(\varepsilon_n-\varepsilon_{n+1})s-
M\int_0^s[(\tilde{\alpha}_{(n+1)0}(u)-\alpha(u))-(\tilde{\alpha}_{n0}(u)-\alpha(u))]du\right.\\
&~~~~~~~~~~~~~~~~~~~~~~~~~~~~~~~~~~~~~~~~~~~~~~~\left.
+\int_0^s(g(u,\tilde{\alpha}_{(n+1)0}(u))-g(u,\tilde{\alpha}_{n0}(u)))dB(u)\right|^2\\
&\leq3|\varepsilon_n-\varepsilon_{n+1}|^2\theta^2+3M^2\left(\int_0^t|\tilde{\alpha}_{(n+1)0}(u)-\tilde{\alpha}_{n0}(u)|du\right)^2+3\sup\limits_{0\leq
s\leq
t}\left|\int_0^s(g(u,\tilde{\alpha}_{(n+1)0}(u))-g(u,\tilde{\alpha}_{n0}(u)))dB(u)\right|^2\\
&\leq3|\varepsilon_n-\varepsilon_{n+1}|^2\theta^2+3M^2\theta\int_0^t|\tilde{\alpha}_{(n+1)0}(u)-\tilde{\alpha}_{n0}(u)|^2du+3\sup\limits_{0\leq
s\leq
t}\left|\int_0^s(g(u,\tilde{\alpha}_{(n+1)0}(u))-g(u,\tilde{\alpha}_{n0}(u)))dB(u)\right|^2.
\end{align*}
Then
\begin{align*}
&E\sup\limits_{0\leq s\leq
t}|\tilde{\alpha}_{(n+1)0}(s)-\tilde{\alpha}_{n0}(s)|^2\\
&\leq3|\varepsilon_n-\varepsilon_{n+1}|^2\theta^2+3M^2\theta
E\int_0^t|\tilde{\alpha}_{(n+1)0}(u)-\tilde{\alpha}_{n0}(u)|^2du+3E\sup\limits_{0\leq
s\leq
t}\left|\int_0^s(g(u,\tilde{\alpha}_{(n+1)0}(u))-g(u,\tilde{\alpha}_{n0}(u)))dB(u)\right|^2\\
&\leq3|\varepsilon_n-\varepsilon_{n+1}|^2\theta^2+3M^2\theta
E\int_0^t|\tilde{\alpha}_{(n+1)0}(u)-\tilde{\alpha}_{n0}(u)|^2du
+12E\int_0^t|g(u,\tilde{\alpha}_{(n+1)0}(u))-g(u,\tilde{\alpha}_{n0}(u))|^2du\\
&\leq3|\varepsilon_n-\varepsilon_{n+1}|^2\theta^2+(3M^2\theta+12L)\int_0^tE|\tilde{\alpha}_{(n+1)0}(u)-\tilde{\alpha}_{n0}(u)|^2du\\
&\leq3|\varepsilon_n-\varepsilon_{n+1}|^2\theta^2+(3M^2\theta+12L)\int_0^tE\sup\limits_{0\leq
s\leq u}|\tilde{\alpha}_{(n+1)0}(s)-\tilde{\alpha}_{n0}(s)|^2du.
\end{align*}
Applying Gronwall's inequality yields
\begin{equation}\label{Ca}
E\sup\limits_{0\leq s\leq
t}|\tilde{\alpha}_{(n+1)0}(s)-\tilde{\alpha}_{n0}(s)|^2\leq
3|\varepsilon_n-\varepsilon_{n+1}|^2\theta^2e^{(3M^2\theta+12L)\theta}.
\end{equation}
The property of $\varepsilon_n$ tells us that there is an $N_1>0$
that for $n\geq N_1$,
\begin{equation}\label{epsiolnC}
\varepsilon_n-\varepsilon_{n+1}\leq\f{1}{\sqrt{3\theta^2e^{(3M^2\theta+12L)\theta}8^n}}.
\end{equation}
Substituting \eqref{epsiolnC} into \eqref{Ca}, we get
\begin{align*}
E\sup\limits_{0\leq s\leq
\theta}|\tilde{\alpha}_{(n+1)0}(s)-\tilde{\alpha}_{n0}(s)|^2\leq
\f{1}{8^n}~~\mbox{for}~~n\geq N_1.
\end{align*}
By Chebyshev's inequality we have
\begin{align*}
\sum\limits_{n=N_1}^{\infty}P\left(\sup\limits_{0\leq s\leq
\theta}|\tilde{\alpha}_{(n+1)0}(s)-\tilde{\alpha}_{n0}(s)|>2^{-(n+1)}\right)&\leq\sum\limits_{n=N_1}^{\infty}\f{E\sup\limits_{0\leq
s\leq
\theta}|\tilde{\alpha}_{(n+1)0}(s)-\tilde{\alpha}_{n0}(s)|^2}{2^{-2(n+1)}}\leq\sum\limits_{n=N_1}^{\infty}4\f{1}{2^n}<\infty.
\end{align*}
In view of the well-known Borel-Cantelli lemma, one sees that for
almost all $\omega\in\Omega$
\begin{equation}\label{alphanlimits}
\sup\limits_{0\leq s\leq
\theta}|\tilde{\alpha}_{(n+1)0}(s)-\tilde{\alpha}_{n0}(s)|\leq
2^{-(n+1)}.
\end{equation}
It tells us that there exists an $N_2(\omega)\geq N_1$, for all
$\omega\in\Omega$ excluding a $P$-null set, for which
\eqref{alphanlimits} holds whenever $n\geq N_2$. Consequently,
$\tilde{\alpha}_{n0}(t)$ uniformly converges to
$\tilde{\alpha}^0(t)$ and $\tilde{\alpha}^0(t)$ is continuous on
$[0, \theta]$ a.s.

Define
\begin{align*}
T_N\triangleq\inf\{t>0:
|\tilde{\alpha}^0(t)|>N~~\mbox{or}~~|\tilde{\alpha}_{n0}(t)|>N\}\wedge
N, ~~\mbox{for every}~~N>0.
\end{align*}
In terms of \eqref{limitalpha} and Lebesgue dominated convergence
theorem, we have
\begin{align*}
&\lim\limits_{n\rightarrow\infty}\int_0^{t\wedge
T_N}f_1^{\varepsilon_n}(s, \tilde{\alpha}_{n0}(s))ds=\int_0^{t\wedge
T_N}f_1(s,
\tilde{\alpha}^0(s))ds~~a.s.\\
&\lim\limits_{n\rightarrow\infty}E\int_0^{t\wedge
T_N}|g(s,\tilde{\alpha}_{n0}(s))-g(s,\tilde{\alpha}^0(s))|^2ds=0,
\end{align*}
which together with Proposition 3.2 in \cite{KaratzasShreve1991}
implies that
\begin{align*}
\lim\limits_{n\rightarrow\infty}\int_0^{t\wedge
T_N}g(s,\tilde{\alpha}_{n0}(s))dB(s)=\int_0^{t\wedge
T_N}g(s,\tilde{\alpha}^0(s))dB(s)~~\mbox{in}~~L^2.
\end{align*}
Hence
\begin{align*}
\tilde{\alpha}^0(t\wedge T_N)=\tilde{\alpha}(0)+\int_0^{t\wedge
T_N}f_1(s, \tilde{\alpha}^0(s))ds+\int_0^{t\wedge
T_N}g(s,\tilde{\alpha}^0(s))dB(s).
\end{align*}
Note that $\lim\limits_{N\rightarrow\infty}T_N=\theta$, then
\begin{align*}
\tilde{\alpha}^0(t)=\tilde{\alpha}(0)+\int_0^{t}f_1(s,
\tilde{\alpha}^0(s))ds+\int_0^{t}g(s,\tilde{\alpha}^0(s))dB(s).
\end{align*}
Therefore, \eqref{limitalphaP} is true. It tells us that for
$\varsigma=\f{1}{2}\min\limits_{t\in[0,\theta]}\{\tilde{\alpha}(t)-\alpha(t)\}>0$,
there is a $\tilde{N}_0$ such that for $n>\tilde{N}_0$,
$|\tilde{\alpha}_{n0}(t)-\tilde{\alpha}(t)|<\varsigma$~~a.s., which
implies that
\begin{equation}\label{tala}
\tilde{\alpha}_{n0}(t)>\alpha(t)~~\mbox{a.s. for}~~ n>\tilde{N}_0.
\end{equation}
Since  $\tilde{\alpha}^0(t)$ is a modification of the solution
$\tilde{\alpha}(t)$ and $\tilde{\alpha}_{n0}(t)$ uniformly converges
to $\tilde{\alpha}^0(t)$, then in order to  verify that
\begin{equation}\label{Aalpha}
\tilde{\alpha}\leq A(\tilde{\alpha})~~\mbox{a.s.},
\end{equation}
we only need to  prove that
\begin{align*}
\tilde{\alpha}_{n0}\leq A(\tilde{\alpha}_{n0})~~\mbox{ a.s.~ for
}~~n>\tilde{N}_0.
\end{align*}
 For this, set $A(\tilde{\alpha}_{n0})=\tilde{\alpha}_{n1}$, where
$\tilde{\alpha}_{n1}$ is the unique solution of \eqref{FZsde} with
$\eta=\tilde{\alpha}_{n0}$. That is
\begin{align*}
d\tilde{\alpha}_{n1}=[f(t,\tilde{\alpha}_{n0})-M(\tilde{\alpha}_{n1}-\tilde{\alpha}_{n0})]dt+g(t,\tilde{\alpha}_{n1})dB(t).
\end{align*}
Set $Y_{n\alpha}(t)=\tilde{\alpha}_{n1}(t)-\tilde{\alpha}_{n0}(t)$.
Then
\begin{align*}
dY_{n\alpha}(t)=[f(t,\tilde{\alpha}_{n0})-M(\tilde{\alpha}_{n1}-\tilde{\alpha}_{n0})-f_1^{\varepsilon_n}(t,\tilde{\alpha}_{n0})]dt+(g(t,\tilde{\alpha}_{n1})-g(t,\tilde{\alpha}_{n0}))dB(t).
\end{align*}
Define the stopping time
\begin{equation}\label{stoppingTime}
\begin{split}
\tau_{\alpha}&\triangleq\inf\{t>0:
\tilde{\alpha}_{n0}(t)>\tilde{\alpha}_{n1}(t), n>\tilde{N}_0\}.
\end{split}
\end{equation}
Obviously,  $\tau_{\alpha}\leq\theta$. In order to verify the
conclusion, we have to show that
\begin{equation}\label{tau}
P(\{\tau_{\alpha}<\theta\})=0.
\end{equation}
For this purpose, let
\begin{align*}
\kappa_{\alpha}\triangleq\inf\{t>\tau_{\alpha}: Z_{n\alpha}(t,
\omega)\triangleq f_1^{\varepsilon_n}(t,
\tilde{\alpha}_{n0}(t))-[f(t,\tilde{\alpha}_{n0}(t))-M(\tilde{\alpha}_{n1}(t)-Y^{-}_{n\alpha}(t)-\tilde{\alpha}_{n0}(t))]>0,
n>\tilde{N}_0\}.
\end{align*}
It is easy to see that $Z_{n\alpha}(t,\omega)$ is
$\mathcal{F}_t$-adapted and its sample path  is continuous. Then
$\kappa_{\alpha}$ is a $\mathcal{F}_t$-stopping time.

We claim that
\begin{equation}\label{kappaStoppingTime}
\kappa_{\alpha}>\tau_{\alpha}~~\mbox{on}~~\{\tau_{\alpha}<\theta\}.
\end{equation}
From the definition of $\kappa_{\alpha}$, we know that
$\kappa_{\alpha}\geq \tau_{\alpha}$ and
\begin{equation}\label{aa}
f_1^{\varepsilon_n}(\kappa_{\alpha},
\tilde{\alpha}_{n0}(\kappa_{\alpha}))-[f(\kappa_{\alpha},\tilde{\alpha}_{n0}(\kappa_{\alpha}))
-M(\tilde{\alpha}_{n1}(\kappa_{\alpha})-Y^{-}_{n\alpha}(\kappa_{\alpha})-\tilde{\alpha}_{n0}(\kappa_{\alpha}))]\geq0,
n>\tilde{N}_0.
\end{equation}
Hence if we show that the case $\kappa_{\alpha}=\tau_{\alpha}$ is
impossible, then  \eqref{kappaStoppingTime} is true. Suppose
$\kappa_{\alpha}=\tau_{\alpha}$ on $\{\tau_{\alpha}<\theta\}$. Since
$\tilde{\alpha}_{n0}(\tau_{\alpha})=\tilde{\alpha}_{n1}(\tau_{\alpha})$,
we have
$Y^{-}_{n\alpha}(\tau_{\alpha})=Y^{-}_{n\alpha}(\kappa_{\alpha})=0$.
This together with \eqref{tala} and the hypothesis  H for $f$
implies
\begin{equation*}\label{first}
\begin{split}
&~~~~f_1^{\varepsilon_n}(\kappa_{\alpha},
\tilde{\alpha}_{n0}(\kappa_{\alpha}))-[f(\kappa_{\alpha},\tilde{\alpha}_{n0}(\kappa_{\alpha}))-M(\tilde{\alpha}_{n1}(\kappa_{\alpha})-Y^{-}_{n\alpha}(\kappa_{\alpha})-\tilde{\alpha}_{n0}(\kappa_{\alpha}))]\\
&=\alpha'(\kappa_{\alpha})-M(\tilde{\alpha}_{n0}(\kappa_{\alpha})-\alpha(\kappa_{\alpha}))-\varepsilon_n-[f(\kappa_{\alpha},\tilde{\alpha}_{n0}(\kappa_{\alpha}))-M(\tilde{\alpha}_{n1}(\kappa_{\alpha})-Y^{-}_{n\alpha}(\kappa_{\alpha})-\tilde{\alpha}_{n0}(\kappa_{\alpha}))]\\
&\leq f(\kappa_{\alpha},
\alpha(\kappa_{\alpha}))-M(\tilde{\alpha}_{n0}(\kappa_{\alpha})-\alpha(\kappa_{\alpha}))-\varepsilon_n-f(\kappa_{\alpha},\tilde{\alpha}_{n0}(\kappa_{\alpha}))\\
&\leq
-M(\alpha(\kappa_{\alpha})-\tilde{\alpha}_{n0}(\kappa_{\alpha}))-M(\tilde{\alpha}_{n0}(\kappa_{\alpha})-\alpha(\kappa_{\alpha}))-\varepsilon_n<0
\end{split}
\end{equation*}
on $\{\tau_{\alpha}<\theta\}$ for  $n>\tilde{N}_0$, which
contradicts \eqref{aa}. Hence \eqref{kappaStoppingTime} holds.
Therefore, it can be seen that for
$s\in[\tau_{\alpha},\kappa_{\alpha}], n>\tilde{N}_0$
\begin{equation}\label{taukappaEstimation}
f_1^{\varepsilon_n}(s,
\tilde{\alpha}_{n0}(s))-[f(s,\tilde{\alpha}_{n0}(s))-M(\tilde{\alpha}_{n1}(s)-Y^{-}_{n\alpha}(s)-\tilde{\alpha}_{n0}(s))]\leq
0
\end{equation}
on $\{\tau_{\alpha}<\theta\}$.

Now we can show that \eqref{tau} is ture. If not, assume that for
some $N$
\begin{align*}
P(\mathcal {B})\triangleq P(\{\tau_{\alpha}<\theta\})>0.
\end{align*}
Since $Y_{n\alpha}(t)$ is a continuous semimartingale
\cite{Mohammed1984}, applying It\^o's formula yields
\begin{equation}\label{EvarphiY}
\begin{split}
&\varphi_{\epsilon}(Y_{n\alpha}((\tau_{\alpha}+t)\wedge
\kappa_{\alpha}\wedge\theta))\\
&=\varphi_{\epsilon}(Y_{n\alpha}(\tau_{\alpha}\wedge
\kappa_{\alpha}\wedge\theta))+\int_{\tau_{\alpha}\wedge
\kappa_{\alpha}\wedge\theta}^{(\tau_{\alpha}+t)\wedge
\kappa_{\alpha}\wedge\theta}\varphi_{\epsilon}'(Y_{n\alpha}(s))[f(s,\tilde{\alpha}_{n0}(s))-M(\tilde{\alpha}_{n1}(s)-\tilde{\alpha}_{n0}(s))-f_1^{\varepsilon_n}(s,
\tilde{\alpha}_{n0}(s))]ds\\
&~~~~+\int_{\tau_{\alpha}\wedge
\kappa_{\alpha}\wedge\theta}^{(\tau_{\alpha}+t)\wedge
\kappa_{\alpha}\wedge\theta}\varphi_{\epsilon}'(Y_{n\alpha}(s))(g(s,\tilde{\alpha}_{n1}(s))-g(s,\tilde{\alpha}_{n0}(s)))dB(s)\\
&~~~~+\f{1}{2}\int_{\tau_{\alpha}\wedge
\kappa_{\alpha}\wedge\theta}^{(\tau_{\alpha}+t)\wedge
\kappa_{\alpha}\wedge\theta}\varphi_{\epsilon}''(Y_{n\alpha}(s))(g(s,\tilde{\alpha}_{n1}(s))-g(s,\tilde{\alpha}_{n0}(s)))^2ds\\
&=\int_{\tau_{\alpha}\wedge
\kappa_{\alpha}\wedge\theta}^{(\tau_{\alpha}+t)\wedge
\kappa_{\alpha}\wedge\theta}\varphi_{\epsilon}'(Y_{n\alpha}(s))[f(s,\tilde{\alpha}_{n0}(s))-M(\tilde{\alpha}_{n1}(s)-\tilde{\alpha}_{n0}(s))-f_1^{\varepsilon_n}(s,
\tilde{\alpha}_{n0}(s))]ds\\
&~~~~+\int_{\tau_{\alpha}\wedge
\kappa_{\alpha}\wedge\theta}^{(\tau_{\alpha}+t)\wedge
\kappa_{\alpha}\wedge\theta}\varphi_{\epsilon}'(Y_{n\alpha}(s))(g(s,\tilde{\alpha}_{n1}(s))-g(s,\tilde{\alpha}_{n0}(s)))dB(s)\\
&~~~~+\f{1}{2}\int_{\tau_{\alpha}\wedge
\kappa_{\alpha}\wedge\theta}^{(\tau_{\alpha}+t)\wedge
\kappa_{\alpha}\wedge\theta}\varphi_{\epsilon}''(Y_{n\alpha}(s))(g(s,\tilde{\alpha}_{n1}(s))-g(s,\tilde{\alpha}_{n0}(s)))^2ds\\
&\triangleq \Delta_1+\Delta_2+\Delta_3,
\end{split}
\end{equation}
where $\varphi_{\epsilon}(y)$ is defined by \eqref{vy}. Note that
$E[\Delta_2|\mathcal {F}_{\tau_{\alpha}}]=0$. This together with the
fact that $I_\mathcal {B}$ is $\mathcal
{F}_{\tau_{\alpha}}$-measurable (see Lemma 1.2.16 in
\cite{KaratzasShreve1991}) implies that
\begin{align*}
E[\Delta_2I_\mathcal {B}]=E[E[\Delta_2I_\mathcal {B}|\mathcal
{F}_{\tau_{\alpha}}]]=E[I_\mathcal {B}E[\Delta_2|\mathcal
{F}_{\tau_{\alpha}}]]=0.
\end{align*}
Multiplying both sides of \eqref{EvarphiY} by the indicator function
$I_\mathcal {B}$, and then taking expectation, we obtain
\begin{align*}
&E[I_\mathcal {B}\varphi_{\epsilon}(Y_{n\alpha}((\tau_{\alpha}+t)\wedge \kappa_{\alpha}\wedge\theta))]\\
&=E\left[I_\mathcal {B}\int_{\tau_{\alpha}\wedge
\kappa_{\alpha}\wedge\theta}^{(\tau_{\alpha}+t)\wedge
\kappa_{\alpha}\wedge\theta}\varphi_{\epsilon}'(Y_{n\alpha}(s))[f(s,\tilde{\alpha}_{n0}(s))-M(\tilde{\alpha}_{n1}(s)-\tilde{\alpha}_{n0}(s))-f_1^{\varepsilon_n}(s,
\tilde{\alpha}_{n0}(s))]ds\right]\\
&~~~~+\f{1}{2}E\left[I_\mathcal {B}\int_{\tau_{\alpha}\wedge
\kappa_{\alpha}\wedge\theta}^{(\tau_{\alpha}+t)\wedge
\kappa_{\alpha}\wedge\theta}\varphi_{\epsilon}''(Y_{n\alpha}(s))(g(s,\tilde{\alpha}_{n1}(s))-g(s,\tilde{\alpha}_{n0}(s)))^2ds\right].
\end{align*}
Letting $\epsilon\rightarrow 0$ yields
\begin{align*}
&E[I_\mathcal {B}(Y^{-}_{n\alpha}((\tau_{\alpha}+t)\wedge \kappa_{\alpha}\wedge\theta))^2]\\
&=E\left[I_\mathcal {B}\int_{\tau_{\alpha}\wedge
\kappa_{\alpha}\wedge\theta}^{(\tau_{\alpha}+t)\wedge
\kappa_{\alpha}\wedge\theta}2(Y^{-}_{n\alpha}(s))[f(s,\tilde{\alpha}_{n0}(s))-M(\tilde{\alpha}_{n1}(s)-\tilde{\alpha}_{n0}(s))-f_1^{\varepsilon_n}(s,
\tilde{\alpha}_{n0}(s))]ds\right]\\
&~~~~+\f{1}{2}E\left[I_\mathcal {B}\int_{\tau_{\alpha}\wedge
\kappa_{\alpha}\wedge\theta}^{(\tau_{\alpha}+t)\wedge
\kappa_{\alpha}\wedge\theta}I_{\{Y_{n\alpha}(s)\leq
0\}}(g(s,\tilde{\alpha}_{n1}(s))-g(s,\tilde{\alpha}_{n0}(s)))^2ds\right]\\
&\leq -2ME\left[I_\mathcal {B}\int_{\tau_{\alpha}\wedge
\kappa_{\alpha}\wedge\theta}^{(\tau_{\alpha}+t)\wedge
\kappa_{\alpha}\wedge\theta}(Y^{-}_{n\alpha}(s))^2ds\right]+\f{L}{2}E\left[I_\mathcal
{B}\int_{\tau_{\alpha}\wedge
\kappa_{\alpha}\wedge\theta}^{(\tau_{\alpha}+t)\wedge
\kappa_{\alpha}\wedge\theta}I_{\{Y_{n\alpha}(s)\leq
0\}}(\tilde{\alpha}_{n1}(s)-\tilde{\alpha}_{n0}(s))^2ds\right]\\
&~~~~+E\left[I_\mathcal {B}\int_{\tau_{\alpha}\wedge
\kappa_{\alpha}\wedge\theta}^{(\tau_{\alpha}+t)\wedge
\kappa_{\alpha}\wedge\theta}2(Y^{-}_{n\alpha}(s))[f(s,\tilde{\alpha}_{n0}(s))-M(\tilde{\alpha}_{n1}(s)-Y^{-}_{n\alpha}(s)-\tilde{\alpha}_{n0}(s))-f_1^{\varepsilon_n}(s,
\tilde{\alpha}_{n0}(s))]ds\right]\\
&\leq -2ME\left[I_\mathcal {B}\int_{\tau_{\alpha}\wedge
\kappa_{\alpha}\wedge\theta}^{(\tau_{\alpha}+t)\wedge
\kappa_{\alpha}\wedge\theta}(Y^{-}_{n\alpha}(s))^2ds\right]+\f{L}{2}E\left[I_\mathcal
{B}\int_{\tau_{\alpha}\wedge
\kappa_{\alpha}\wedge\theta}^{(\tau_{\alpha}+t)\wedge
\kappa_{\alpha}\wedge\theta}I_{\{Y_{\alpha}(s)\leq
0\}}(\tilde{\alpha}_{n1}(s)-\tilde{\alpha}_{n0}(s))^2ds\right]\\
&=\left(\f{L}{2}-2M\right)E\left[I_\mathcal
{B}\int_{\tau_{\alpha}\wedge
\kappa_{\alpha}\wedge\theta}^{(\tau_{\alpha}+t)\wedge
\kappa_{\alpha}\wedge\theta}(Y^{-}_{n\alpha}(s))^2ds\right]\\
&=\left(\f{L}{2}-2M\right)E\left[I_\mathcal
{B}\int_{0}^{t}(Y^{-}_{n\alpha}((\tau_{\alpha}+s)\wedge
\kappa_{\alpha}\wedge\theta))^2ds\right]\leq 0,
\end{align*}
where the second inequality holds by $Y^{-}_{n\alpha}(s)\leq 0$ and
$f(s,\tilde{\alpha}_{n0}(s))-M(\tilde{\alpha}_{n1}(s)-Y^{-}_{n\alpha}(s)-\tilde{\alpha}_{n0}(s))-f_1^{\varepsilon_n}(s,
\tilde{\alpha}_{n0}(s))\geq 0$ for $s\in [\tau_{\alpha},
\kappa_{\alpha}], n>\tilde{N}_0$,  the last  inequality is by
$\f{L}{2}-2M<0$. Therefore
\begin{align*}
E[I_\mathcal {B}(Y^{-}_{n\alpha}((\tau_{\alpha}+t)\wedge
\kappa_{\alpha}\wedge\theta))^2]=0,
\end{align*}
which tells us that
\begin{align*}
\tilde{\alpha}_{n1}((\tau_{\alpha}+t)\wedge
\kappa_{\alpha}\wedge\theta))\geq\tilde{\alpha}_{n0}((\tau_{\alpha}+t)\wedge
\kappa_{\alpha}\wedge\theta))~~a.s.
\end{align*}
for every $t\geq 0, n>\tilde{N}_0$  on $\mathcal {B}$. It follows
from the continuity of $\tilde{\alpha}_{n0}(t),
\tilde{\alpha}_{n1}(t)$ that
\begin{align*}
\tilde{\alpha}_{n1}((\tau_{\alpha}+t)\wedge
\kappa_{\alpha}\wedge\theta))\geq\tilde{\alpha}_{n0}((\tau_{\alpha}+t)\wedge
\kappa_{\alpha}\wedge\theta)),t\geq 0,n>\tilde{N}_0 ~~a.s.
\end{align*}
on $\mathcal {B}$. This contradicts \eqref{stoppingTime}, which
shows that \eqref{tau} holds. Hence we have
$P(\{\tau_{\alpha}=\theta\})=1$. Therefore
\begin{align*}
P(\{\tilde{\alpha}_{n0}(t)\leq\tilde{\alpha}_{n1}(t), t\geq 0,
n>\tilde{N}_0\})=1,
\end{align*}
i.e.
\begin{align*}
\tilde{\alpha}_{n0}\leq A(\tilde{\alpha}_{n0}),
n>\tilde{N}_0~~\mbox{a.s.}
\end{align*}

  Similarly, we can
prove $\tilde{\beta}\geq A(\tilde{\beta})$~~a.s.

To prove (ii), suppose that $u_1=A(\eta_1)$ and $u_2=A(\eta_2)$. In
order to get $u_2\geq u_1$ a.s., we consider the auxiliary system:
\begin{equation}\label{un1auxiliarysystem}
du_{n1}(t)=[f(t,\eta_1)-\varepsilon_n-M(u_{n1}-\eta_1)]dt+g(t,u_{n1})dB(t)
\end{equation}
with initial value $u_{n1}(0)=u_1(0)$, where $\varepsilon_n$ is
defined as previous.

Setting $Y_{nu}(t)=u_2(t)-u_{n1}(t)$, we get
\begin{align*}
dY_{nu}(t)=[(f(t,\eta_2)-M(u_2-\eta_2))-(f(t,\eta_1)-\varepsilon_n-M(u_{n1}-\eta_1))]dt+(g(t,u_2)-g(t,u_{n1}))dB(t).
\end{align*}
Define the stopping time
\begin{equation}\label{ustoppingTime}
\begin{split}
\tau_u&\triangleq\inf\{t>0: u_{n1}(t)>u_2(t)\}.
\end{split}
\end{equation}
It is clear that $\tau_{u}\leq\theta$. In order to verify the
conclusion, we have to show that
\begin{equation}\label{taueta}
P(\{\tau_{u}<\theta\})=0.
\end{equation}
Let
\begin{align*}
\kappa_{u}\triangleq\inf\{t>\tau_{u}: Z_{nu}(t, \omega)\triangleq
[&f(t,
\eta_1(t))-\varepsilon_n-M(u_{n1}(t)-\eta_1(t))]\\
&-[f(t,\eta_2(t))-M(u_2(t)-Y^{-}_{nu}(t)-\eta_2(t))]>0\},
\end{align*}
where $Z_{nu}(t,\omega)$ is $\mathcal{F}_t$-adapted and  it is
continuous for a fixed $\omega$. Hence $\kappa_{u}$ is an
$\mathcal{F}_t$ stopping time.

We claim that
\begin{equation}\label{ukappaStoppingTime}
\kappa_{u}>\tau_{u}~~\mbox{on}~~\{\tau_{u}<\theta\}.
\end{equation}
By the definition of $\kappa_{u}$, we know that $\kappa_{u}\geq
\tau_{u}$ and
\begin{equation}\label{uaa}
[f(\kappa_{u},
\eta_1(\kappa_{u}))-\varepsilon_n-M(u_{n1}(\kappa_{u})-\eta_1(\kappa_{u}))]-[f(\kappa_{u},\eta_2(\kappa_{u}))-M(u_2(\kappa_{u})-Y^{-}_{nu}(\kappa_{u})-\eta_2(\kappa_{u}))]\geq0.
\end{equation}
Then in order to prove \eqref{ukappaStoppingTime}, we only need to
prove that the case $\kappa_{u}=\tau_{u}$ is impossible. If not,
suppose $\kappa_{u}=\tau_{u}$ on $\{\tau_{u}<\theta\}$. Since
$u_{n1}(\tau_{u})=u_2(\tau_{u})$, we have
$Y^{-}_{nu}(\tau_{u})=Y^{-}_{nu}(\kappa_{u})=0$. And then
\begin{align*}
&[f(\kappa_{u},
\eta_1(\kappa_{u}))-\varepsilon_n-M(u_{n1}(\kappa_{u})-\eta_1(\kappa_{u}))]-[f(\kappa_{u},\eta_2(\kappa_{u}))-M(u_2(\kappa_{u})-Y^{-}_{nu}(\kappa_{u})-\eta_2(\kappa_{u}))]\\
&=f(\kappa_{u},
\eta_1(\kappa_{u}))-f(\kappa_{u},\eta_2(\kappa_{u}))+M(\eta_1(\kappa_{u})-\eta_2(\kappa_{u}))-\varepsilon_n\\
&\leq-M(\eta_1(\kappa_{u})-\eta_2(\kappa_{u}))+M(\eta_1(\kappa_{u})-\eta_2(\kappa_{u}))-\varepsilon_n=-\varepsilon_n<0
\end{align*}
on $\{\tau_{u}<\theta\}$, which contradicts \eqref{uaa}. Hence
\eqref{ukappaStoppingTime} holds. Therefore, it can be seen that for
$s\in[\tau_{u},\kappa_{u}]$
\begin{equation}\label{utaukappaEstimation}
[f(s,
\eta_1(s))-\varepsilon_n-M(u_{n1}(s)-\eta_1(s))]-[f(s,\eta_2(s))-M(u_2(s)-Y^{-}_{nu}(s)-\eta_2(s))]\leq
0
\end{equation}
on $\{\tau_{u}<\theta\}$.

Now we prove \eqref{taueta}. If not,
\begin{align*}
P(\mathcal {C})\triangleq P(\{\tau_{u}<\theta\})>0.
\end{align*}
By It\^o's formula, we get
\begin{equation}\label{uEvarphiY}
\begin{split}
&\varphi_{\epsilon}(Y_{nu}((\tau_{u}+t)\wedge
\kappa_{u}\wedge\theta))=\varphi_{\epsilon}(Y_{nu}(\tau_{u}\wedge
\kappa_{u}\wedge\theta))\\
&~~~~+\int_{\tau_{u}\wedge
\kappa_{u}\wedge\theta}^{(\tau_{u}+t)\wedge
\kappa_{u}\wedge\theta}\varphi_{\epsilon}'(Y_{nu}(s))[f(s,\eta_2(s))-M(u_2(s)-\eta_2(s))-f(s,
\eta_1(s))+\varepsilon_n+M(u_{n1}(s)-\eta_1(s))]ds\\
&~~~~+\int_{\tau_{u}\wedge
\kappa_{u}\wedge\theta}^{(\tau_{u}+t)\wedge
\kappa_{u}\wedge\theta}\varphi_{\epsilon}'(Y_{nu}(s))(g(s,u_2(s))-g(s,u_{n1}(s)))dB(s)\\
&~~~~+\f{1}{2}\int_{\tau_{u}\wedge
\kappa_{u}\wedge\theta}^{(\tau_{u}+t)\wedge
\kappa_{u}\wedge\theta}\varphi_{\epsilon}''(Y_{nu}(s))(g(s,u_2(s))-g(s,u_{n1}(s)))^2ds\\
&=\int_{\tau_{u}\wedge \kappa_{u}\wedge\theta}^{(\tau_{u}+t)\wedge
\kappa_{u}\wedge\theta}\varphi_{\epsilon}'(Y_{nu}(s))[f(s,\eta_2(s))-M(u_2(s)-\eta_2(s))-f(s,
\eta_1(s))+\varepsilon_n+M(u_{n1}(s)-\eta_1(s))]ds\\
&~~~~+\int_{\tau_{u}\wedge
\kappa_{u}\wedge\theta}^{(\tau_{u}+t)\wedge
\kappa_{u}\wedge\theta}\varphi_{\epsilon}'(Y_{nu}(s))(g(s,u_2(s))-g(s,u_{n1}(s)))dB(s)\\
&~~~~+\f{1}{2}\int_{\tau_{u}\wedge
\kappa_{u}\wedge\theta}^{(\tau_{u}+t)\wedge
\kappa_{u}\wedge\theta}\varphi_{\epsilon}''(Y_{nu}(s))(g(s,u_2(s))-g(s,u_{n1}(s)))^2ds\\
&\triangleq \tilde{\Delta}_1+\tilde{\Delta}_2+\tilde{\Delta}_3.
\end{split}
\end{equation}
It is easy to verify that
\begin{align*}
E[\tilde{\Delta}_2I_\mathcal {C}]=E[E[\tilde{\Delta}_2I_\mathcal
{C}|\mathcal {F}_{\tau_{u}}]]=E[I_\mathcal
{C}E[\tilde{\Delta}_2|\mathcal {F}_{\tau_{u}}]]=0.
\end{align*}
Multiplied by the indicator function $I_\mathcal {C}$ to the both
sides of \eqref{uEvarphiY}, and then taking expectation, we obtain
\begin{align*}
&E[I_\mathcal {C}\varphi_{\epsilon}(Y_{nu}((\tau_{u}+t)\wedge \kappa_{u}\wedge\theta))]\\
&=E\left[I_\mathcal {C}\int_{\tau_{u}\wedge
\kappa_{u}\wedge\theta}^{(\tau_{u}+t)\wedge
\kappa_{u}\wedge\theta}\varphi_{\epsilon}'(Y_{nu}(s))[f(s,\eta_2(s))-M(u_2(s)-\eta_2(s))-f(s,
\eta_1(s))+\varepsilon_n+M(u_{n1}(s)-\eta_1(s))]ds\right]\\
&~~~~+\f{1}{2}E\left[I_\mathcal {C}\int_{\tau_{u}\wedge
\kappa_{u}\wedge\theta}^{(\tau_{u}+t)\wedge
\kappa_{u}\wedge\theta}\varphi_{\epsilon}''(Y_{nu}(s))(g(s,u_2(s))-g(s,u_{n1}(s)))^2ds\right].
\end{align*}
Letting $\epsilon\rightarrow 0$ yields
\begin{align*}
&E[I_\mathcal {C}(Y^{-}_{nu}((\tau_{u}+t)\wedge \kappa_{u}\wedge\theta))^2]\\
&=E\left[I_\mathcal {C}\int_{\tau_{u}\wedge
\kappa_{u}\wedge\theta}^{(\tau_{u}+t)\wedge
\kappa_{u}\wedge\theta}2(Y^{-}_{nu}(s))[f(s,\eta_2(s))-M(u_2(s)-\eta_2(s))-f(s,
\eta_1(s))+\varepsilon_n+M(u_{n1}(s)-\eta_1(s))]ds\right]\\
&~~~~+\f{1}{2}E\left[I_\mathcal {C}\int_{\tau_{u}\wedge
\kappa_{u}\wedge\theta}^{(\tau_{u}+t)\wedge
\kappa_{u}\wedge\theta}I_{\{Y_{nu}(s)\leq
0\}}(g(s,u_2(s))-g(s,u_{n1}(s)))^2ds\right]\\
&\leq -2ME\left[I_\mathcal {C}\int_{\tau_{u}\wedge
\kappa_{u}\wedge\theta}^{(\tau_{u}+t)\wedge
\kappa_{u}\wedge\theta}(Y^{-}_{nu}(s))^2ds\right]+\f{L}{2}E\left[I_\mathcal
{C}\int_{\tau_{u}\wedge \kappa_{u}\wedge\theta}^{(\tau_{u}+t)\wedge
\kappa_{u}\wedge\theta}I_{\{Y_{nu}(s)\leq
0\}}(u_2(s)-u_{n1}(s))^2ds\right]\\
&~~~~+E\left[I_\mathcal {C}\int_{\tau_{u}\wedge
\kappa_{u}\wedge\theta}^{(\tau_{u}+t)\wedge
\kappa_{u}\wedge\theta}2(Y^{-}_{nu}(s))[f(s,\eta_2(s))-M(u_2(s)-Y^{-}_u(s)-\eta_2(s))\right.\\
&~~~~~~~~~~~~~~~~~~~~~~~~~~~~~~~~~~~~~~~~~~~~~~~~~~~~~\left.-f(s,
\eta_1(s))+\varepsilon_n+M(u_{n1}(s)-\eta_1(s))]ds\right]\\
&\leq -2ME\left[I_\mathcal {C}\int_{\tau_{u}\wedge
\kappa_{u}\wedge\theta}^{(\tau_{u}+t)\wedge
\kappa_{u}\wedge\theta}(Y^{-}_{nu}(s))^2ds\right]+\f{L}{2}E\left[I_\mathcal
{C}\int_{\tau_{u}\wedge \kappa_{u}\wedge\theta}^{(\tau_{u}+t)\wedge
\kappa_{u}\wedge\theta}I_{\{Y_{nu}(s)\leq
0\}}(u_2(s)-u_{n1}(s))^2ds\right]\\
&=\left(\f{L}{2}-2M\right)E\left[I_\mathcal {C}\int_{\tau_{u}\wedge
\kappa_{u}\wedge\theta}^{(\tau_{u}+t)\wedge
\kappa_{u}\wedge\theta}(Y^{-}_{nu}(s))^2ds\right]\\
&=\left(\f{L}{2}-2M\right)E\left[I_\mathcal
{C}\int_{0}^{t}(Y^{-}_{nu}((\tau_{u}+s)\wedge
\kappa_{u}\wedge\theta))^2ds\right]\leq 0,
\end{align*}
where the last but one inequality holds by $Y^{-}_{nu}(s)\leq 0$ and
$f(s,\eta_2(s))-M(u_2(s)-Y^{-}_{nu}(s)-\eta_2(s))-f(s,
\eta_1(s))+\varepsilon_n+M(u_{n1}(s)-\eta_1(s))\geq 0$ for $s\in
[\tau_{u}, \kappa_{u}]$, the last inequality is by $\f{L}{2}-2M<0$.
Therefore
\begin{align*}
E[I_\mathcal {C}(Y^{-}_{u}((\tau_{u}+t)\wedge
\kappa_{u}\wedge\theta))^2]=0,
\end{align*}
which tells us that
\begin{align*}
u_2((\tau_{u}+t)\wedge \kappa_{u}\wedge\theta))\geq
u_{n1}((\tau_{u}+t)\wedge \kappa_{u}\wedge\theta))~~a.s.
\end{align*}
for every $t\geq 0$ on $\mathcal {C}$. It follows from the
continuity of $u_{n1}(t), u_2(t)$ that
\begin{align*}
u_2((\tau_{u}+t)\wedge \kappa_{u}\wedge\theta))\geq
u_{n1}((\tau_{u}+t)\wedge \kappa_{u}\wedge\theta)),t\geq 0~~a.s.
\end{align*}
on $\mathcal {C}$. This contradicts \eqref{ustoppingTime}, which
shows $P(\mathcal {C})= P(\{\tau_{u}<\theta\})=0$. Hence we have
$P(\{\tau_{u}=\theta\})=1$. Therefore
\begin{align*}
P(\{u_{n1}(t)\leq u_2(t), t\geq 0\})=1.
\end{align*}
On the other hand, from \eqref{un1auxiliarysystem} and stochastic
comparison theorem, we have
\begin{align*}
u_{11}(t)\leq u_{21}(t)\leq \cdots \leq u_{n1}(t)\leq
u_1(t)~~\mbox{a.s.}~~\mbox{for all}~~t\geq 0.
\end{align*}
Define
\begin{align*}
\lim\limits_{n\rightarrow\infty}u_{n1}(t)\triangleq\tilde{u}_1(t)~~\mbox{a.s.}~~\mbox{for
all}~~t\geq 0.
\end{align*}
 As in the
previous proof, we can show that $u_{n1}(t)$  uniformly converges to
$\tilde{u}_1(t)$ on $t\in[0,T]$ a.s. as $n\rightarrow\infty$. And
thus $\tilde{u}_1(t)$ satisfies\begin{align*}
d\tilde{u}_{1}(t)=[f(t,\eta_1)-M(\tilde{u}_{1}-\eta_1)]dt+g(t,\tilde{u}_{1})dB(t)~~\mbox{a.s.}~~\mbox{for
every}~~t\in [0,\theta]
\end{align*}
with initial value $\tilde{u}_{1}(0)=u_1(0)$. Therefore, by the
uniqueness of strong solutions  we get that $\tilde{u}_1(t)$ is a
modification of the solution $u_1(t)$.

Consequently
\begin{align*}
P(\{u_{1}(t)\leq u_2(t), t\geq 0\})=1,
\end{align*}
i.e.
\begin{align*}
A(\eta_1)\leq A(\eta_2)~~\mbox{a.s.}
\end{align*}

{\bf Step 3:} It is now easy to define the sequences
$\{\tilde{\alpha}_n(t)\}, \{\tilde{\beta}_n(t)\}$ with
$\tilde{\alpha}=\tilde{\alpha}_0, \tilde{\beta}=\tilde{\beta}_0$
such that $\tilde{\alpha}_n=A(\tilde{\alpha}_{n-1}),
\tilde{\beta}_n=A(\tilde{\beta}_{n-1})$, we can conclude
\begin{align*}
\tilde{\alpha}=\tilde{\alpha}_0\leq \tilde{\alpha}_1\leq \cdots\leq
\tilde{\alpha}_n \leq \tilde{\beta}_n\leq \cdots \leq
\tilde{\beta}_1 \leq \tilde{\beta}_0=\tilde{\beta}~~\mbox{a.s.}
\end{align*}

Using the monotone convergence theorem yields
$\lim\limits_{n\rightarrow\infty}\tilde{\alpha}_n(t)=a(t),
\lim\limits_{n\rightarrow\infty}\tilde{\beta}_n(t)=b(t)$ uniformly
a.s. This implies that
\begin{align*}
da(t)&=f(t,a(t))dt+g(t,a(t))dB(t),\\
db(t)&=f(t,b(t))dt+g(t,b(t))dB(t).
\end{align*}
Therefore, there is a  solution $x_0^*(t)\in[a(t),b(t)]$ of system
\eqref{sde} for $t\in[0,\theta]$ and
$x_0^*(0)\overset{d}{=}x_0^*(\theta)$.

Now we prove the second result of this theorem. For $t\in[\theta,
2\theta]$, let $s=t-\theta, s\in[0,\theta]$. Consider
\begin{equation}\label{sdeG} dx_1(s)=
f(s,x_1(s))ds+g(s,x_1(s))d\widetilde{B}(s), s\in[0,\theta]
\end{equation}
with initial value $x_1(0)=x_0^*(\theta)$, where
$\widetilde{B}(s)=B(s+\theta)-B(\theta)$ has the same distribution
as $B(s)$. From the uniqueness of the weak solution, we can know
$x_1(s)=x_0^*(s)\in[a(t),b(t)]$ is the solution of system
\eqref{sdeG} for $s\in[0,\theta]$ and
$x_1(0)=x_0^*(\theta)\overset{d}{=}x_1(\theta)\overset{d}{=}x_0^*(0)$.
 Since $f(t)$
and $g(t)$ are $\theta$-periodic, $x_1^*(t)\triangleq
x_0^*(t-\theta), t\in[\theta, 2\theta]$ is still a solution of
\begin{equation}\label{sdeGG}
dx(t)= f(t,x(t))dt+g(t,x(t))d{B}(t),
\end{equation}
and $x_0^*(\theta)=x_1^*(\theta)\overset{d}{=}x_1^*(2\theta)$.
Reaping this process, we can obtain a sequence $\{x_k^*(t)\},
t\in[k\theta,(k+1)\theta],k\in \mathbb{Z}$, and
$x_{k-1}^*(k\theta)=x_k^*(k\theta)\overset{d}{=}x_k^*((k+1)\theta)$.
Obviously, they are the same in distribution. Therefore, by the
uniqueness, $x^*(t)\triangleq x_k^*(t),
t\in[k\theta,(k+1)\theta],k\in \mathbb{Z}$ is the solution of system
\eqref{sde}, and it is $\theta-$periodic in distribution.

\end{proof}

\section{Multi-dimensional SDEs}
Consider the following $d$-dimensional SDE
\begin{equation}\label{dsde}
dX(t)=\mathbbm{f}(t,X(t))dt+\mathbbm{g}(t,X(t))dB(t),
\end{equation}
where $X(t)=(X_1(t),X_2(t),\cdots,X_d(t))^\top,
B(t)=(B_1(t),B_2(t),\cdots,B_r(t))^\top$ is an $r$-dimensional
${\cal{F}}_t$-adapted Gaussian process with values in
$\mathbb{R}^r$, and $B_1(t),B_2(t),\cdots,B_r(t)$ are mutually
independent. Assume the drift term
 and the
diffusion term
$\mathbbm{f}:\mathbb{R}_+\times\mathbb{R}^d\rightarrow\mathbb{R}^d,
\mathbbm{g}:\mathbb{R}_+\times\mathbb{R}^d\rightarrow\mathbb{R}^{d\times
r}$ are continuous and satisfy
\begin{align*}
{\bf{H^*}}:~~
&\mathbbm{f}(t+\theta,x)=\mathbbm{f}(t,x), \mathbbm{g}(t+\theta,x)=\mathbbm{g}(t,x),\\
&\mathbbm{f}_i(t,x)-\mathbbm{f}_i(t,y)\geq-M (x_i-y_i), i=1,2,\cdots,d, t\in[0,\theta]~\mbox{for any}~x,y~\mbox{such that}~\alpha\leq y\leq x\leq\beta,\\
&\|\mathbbm{f}(t,x)-\mathbbm{f}(t,y)\|\leq M\|x-y\|,~\mbox{for
any}~x,y\in[\alpha, \beta],\\
&\sum\limits_{j=1}^r|\mathbbm{g}_{ij}(t,x)-\mathbbm{g}_{ij}(t,y)|^2\leq
L|x_i-y_i|^2, i=1,2,\cdots,d~\mbox{for any}~x,y\in[\alpha, \beta],
\end{align*}
where  $\theta, M$ and $L$ are positive constants,
$\alpha(t),\beta(t)$ are defined by Definition \ref{uls} with
$h(t,x)=\mathbbm{f}(t,x)$.

Define two SDEs by
\begin{equation}\label{dalpha}
\begin{split}
d\bar{\alpha}_i(t)&=[\alpha'_i(t)-M(\bar{\alpha}_i(t)-\alpha_i(t))]dt+\sum\limits_{j=1}^{r}\mathbbm{g}_{ij}(t,\bar{\alpha}(t))dB_j(t)\\
&:=\mathbbm{f}_{1i}(t,\bar{\alpha}(t))dt+\sum\limits_{j=1}^{r}\mathbbm{g}_{ij}(t,\bar{\alpha}(t))dB_j(t),~~i=1,2,\cdots,d,
\end{split}
\end{equation}
\begin{equation}\label{dbeta}
\begin{split}
 d\bar{\beta}_i(t)&=[\beta'_i(t)+M(\bar{\beta}_i(t)-\beta_i(t))]dt+\sum\limits_{j=1}^{r}\mathbbm{g}_{ij}(t,\bar{\beta}(t))dB_j(t)\\
 &:=\mathbbm{f}_{2i}(t,\bar{\beta}(t))dt+\sum\limits_{j=1}^{r}\mathbbm{g}_{ij}(t,\bar{\beta}(t))dB_j(t),~~i=1,2,\cdots,d
 \end{split}
\end{equation}
with initial values $\bar{\alpha}(0)=\alpha(0)+\zeta$ and
$\bar{\beta}(0)=\beta(0)-\zeta$, respectively and $t\in[0,\theta]$,
where
$\bar{\alpha}(t)=(\bar{\alpha}_1(t),\bar{\alpha}_2(t),\cdots,\bar{\alpha}_d(t))^\top,
\bar{\beta}(t)=(\bar{\beta}_1(t),\bar{\beta}_2(t),\cdots,\bar{\beta}_d(t))^\top$
and $\alpha(t),\beta(t)$ are defined by Definition \ref{uls} with
$h(t,x)=\mathbbm{f}(t,x)$, $\zeta\in\mathbb{R}^d_+$ and $\|\zeta\|$
is sufficiently small (here $\|\cdot\|$ is the general Euclidean
norm).

 As the same argument as in Section 3, we can obtain the following results.

\begin{lemma}\label{abb}
 Let $\bar{\alpha}(t), \bar{\beta}(t)~~(t\in[0,\theta])$ be
solutions of equations \eqref{dalpha} and \eqref{dbeta} with initial
values $\bar{\alpha}(0)$ and $\bar{\beta}(0)$, respectively. Assume
the second condition in hypothesis $H^*$ is satisfied. Then
\begin{align*}
\alpha(t)\leq\bar{\alpha}(t)\leq\bar{\beta}(t)\leq\beta(t)~~a.s.
\end{align*}
\end{lemma}

We now state the second main result as follows:

\begin{theorem}\label{dth}
Let hypothesis $H^*$ hold, and assume there are strict lower and
upper solutions $\alpha$ and $\beta$ for $h=\mathbbm{f}$ with
$\alpha<\beta$. Then there exit monotone sequences
$\{\bar{\alpha}_n(t)\}, \{\bar{\beta}_n(t)\}$ with
$\alpha_0=\bar{\alpha}, \beta_0=\bar{\beta}$ such that
$\lim\limits_{n\rightarrow\infty}\bar{\alpha}_n(t)=\bar{a}(t),
\lim\limits_{n\rightarrow\infty}\bar{\beta}_n(t)=\bar{b}(t)$ and
\begin{align*}
\bar{\alpha}=\bar{\alpha}_0\leq\bar{\alpha}_1\leq\cdots\leq\bar{\alpha}_n\leq
a\leq u\leq
b\leq\bar{\beta}_n\leq\cdots\leq\bar{\beta}_1\leq\bar{\beta}_0=\bar{\beta}
\end{align*}
on $[0,\theta]$, where $u$ is a solution of system \eqref{dsde} such
that $\bar{\alpha}(t)\leq u(t)\leq\bar{\beta}(t)$ on $[0,\theta]$
a.s., and $u(0)\overset{d}{=}u(\theta)$. Therefore, there is a $\theta-$periodic solution in distribution of system
\eqref{dsde}.
\end{theorem}

\section{Applications}
In this section, we give some examples to illustrate our results
developed in this paper.

\begin{example}\label{E1}
Consider the following scalar SDE:
\begin{equation}\label{e1}
dx=\left(-a(t)x^{2n+1}+\sum\limits_{i=1}^{2n}a_i(t)x^i+e(t)\right)dt+xdB(t).
\end{equation}
Here $a, a_i, e: \mathbb{R}\rightarrow\mathbb{R} $ are continuous
$1-$periodic functions
 ($i=1,2,\cdots,2n$) and $a(t)\geq\sigma>0$, B(t) is a one-dimensional Gaussian process.
 Let
\begin{align*}
\alpha(t)=c\left(-1+\f{t}{2}\right),
~~\beta(t)=c\left(1-\f{t}{2}\right), ~~t\in[0,1],~~c\gg 1.
\end{align*}
Obviously, $\alpha(t)<\beta(t)$ for $t\in[0,1]$. Besides, it is easy
to check that $\alpha(0)\leq \alpha(1)$, $\beta(0)\geq\beta(1)$, and
\begin{align*}
\alpha'=\f{c}{2}< f(t,\alpha),~~\beta'=-\f{c}{2}> f(t,\beta),
\end{align*}
when $c$ is sufficiently large. Moreover, we have
\begin{align*}
|f(t,x)-f(t,y)|&=\left|\left(-a(t)x^{2n+1}+\sum\limits_{i=1}^{2n}a_i(t)x^i+e(t)\right)-\left(-a(t)y^{2n+1}+\sum\limits_{i=1}^{2n}a_i(t)y^i+e(t)\right)\right|\\
&=|x-y|\left|-a(t)\left(x^{2n}+x^{2n-1}y+\cdots+y^{2n}\right)+\sum\limits_{i=1}^{2n}a_i(t)\left(x^{i-1}+x^{i-2}y+\cdots+y^i\right)\right|\\
&\leq M|x-y|~~\mbox{for any}~~ x,y\in[\alpha,\beta], ~~~t\in[0,1],
\end{align*}
where $M=\max\limits_{t\in[0,1],
x,y\in[\alpha,\beta]}\left|-a(t)\left(x^{2n}+x^{2n-1}y+\cdots+y^{2n}\right)+\sum\limits_{i=1}^{2n}a_i(t)\left(x^{i-1}+x^{i-2}y+\cdots+y^i\right)\right|<\infty$,
and
\begin{align*}
|g(t,x)-g(t,y)|^2=|x-y|^2~~\mbox{for any}~~
x,y\in[\alpha,\beta],t\in[0,1].
\end{align*}
Therefore, by Theorems \ref{1th}, there is a 1-periodic
solution $x(t)$ in distribution of system \eqref{e1}.
\end{example}

\begin{example}\label{E2}
Consider the following $d$-dimensional SDE:
\begin{equation}\label{e2}
dX=\left(-A(t)X+p(t)\right)dt+XdB(t),
\end{equation}
where $X=(x_1,x_2,\cdots,x_d)^\top$,
$B(t)=(B_1(t),B_2(t),\cdots,B_d(t))$ are a $d-$dimensional Gaussian
process, $A=(a_{ij})_{d\times d}:
\mathbb{R}\rightarrow\mathbb{R}^{d\times d}$,
$p=(p_1,p_2,\cdots,p_d)^\top:\mathbb{R}\rightarrow\mathbb{R}^d$ are
continuous $\theta-$periodic functions
 and satisfy
\begin{align*}
&a_{ij}\leq 0,~~i\neq j,\\
&a_{ii}\geq\sigma>0,~~i,j=1,2,\cdots, d,\\
&a_{ii}\geq-\sum\limits_{j\neq i}a_{ij},~~i=1,2,\cdots, d.
\end{align*}
Set
\begin{align*}
\alpha(t)=-c(1,1,\cdots,1)^\top, ~~\beta(t)=-\alpha(t),
~~t\in[0,\theta],~~c\gg 1.
\end{align*}
Obviously, $\alpha(t)<\beta(t)$ for $t\in[0,\theta]$, and
$\alpha(0)\leq \alpha(1)$, $\beta(0)\geq\beta(1)$. Besides, it is
easy to check that
\begin{align*}
&\alpha'=(0,0,\cdots,0)^\top < \mathbbm{f}(t,\alpha)
=\left(c\sum\limits_{j=1}^d a_{1j} +p_1(t), c\sum\limits_{j=1}^d
a_{2j} +p_2(t), \cdots, c\sum\limits_{j=1}^d a_{dj} +p_d(t)
\right)^\top,\\
&\beta'=(0,0,\cdots,0)^\top>
\mathbbm{f}(t,\beta)=\left(-c\sum\limits_{j=1}^d a_{1j} +p_1(t),
-c\sum\limits_{j=1}^d a_{2j} +p_2(t), \cdots, -c\sum\limits_{j=1}^d
a_{dj}+p_d(t)\right)^\top,
\end{align*}
for $c$ sufficiently large. Moreover, we have for any
$\alpha\leq Y\leq X\leq\beta, t\in[0,\theta]$,
\begin{align*}
\mathbbm{f}_i(t,X)-\mathbbm{f}_i(t,Y)&=-\sum_{j=1}^na_{ij}(x_j-y_j)
\geq-a_{ii}(x_i-y_i), i=1,2,\cdots, d,
\end{align*}
and for any $X,Y\in[\alpha,\beta],t\in[0,\theta]$,
\begin{align*}
\|\mathbbm{f}(t,X)-\mathbbm{f}(t,Y)\|&=\|A(t)(X-Y)\|\leq M\|X-Y\|,\\
|\mathbbm{g}_i(t,X)-\mathbbm{g}_i(t,Y)|^2&=|x_i-y_i|^2,
i=1,2,\cdots, d,
\end{align*}
where $M=\max\limits_{t\in[0,\theta]}\{\|A(t)\|\}$. Therefore, by
Theorems \ref{dth}, there is a $\theta$-periodic
solution $X(t)$ in distribution of system \eqref{e1}.
\end{example}

\section*{Acknowledgments}
The first author was supported by NSFC grant 11601043, China
Postdoctoral Science Foundation (Grant No. 2016M590243,
2019T120226), Jiangsu Province ``333 High-Level Personnel Training
Project'' (Grant No. BRA2017468) and Qing Lan Project of Jiangsu
Province of 2016 and 2017. The second author was supported by NSFC
grant 11201173. The third author was supported by National Basic
Research Program of China (Grant No. 2013CB834100) and NSFC grants
11171132 and 11571065.


\begin{thebibliography}{10}\small
\bibitem{BaiJiang2017}
X.M. Bai, J.F. Jiang, Comparison theorem for stochastic functional
differential equations and applications. J. Dynam. Differential
Equations 29 (2017) 1-24.

\bibitem{BernfeldLakshmtkantham1974}
S.R. Bernfeld, V. Lakshmtkantham, An Introduction to Nonlinear
Boundary Value Problems. Mathematics in Science and Engineering 109,
Academic Press, 1974.


\bibitem{BuckdahnPend1999}
R. Buckdahn, S. Peng, Ergodic backward stochastic differential
equations and associated partial differential equations. Prog.
Probab. 45 (1999) 73-85.

\bibitem{CandidoLlibreNovaes2017}
M. Candido,  J. Llibre,  D. Novaes,  Persistence of periodic
solutions for higher order perturbed differential systems via
Lyapunov-Schmidt reduction. Nonlinearity 30 (2017) 3560-3586.

\bibitem{ChenHanLiYang2017}
F. Chen, Y. Han, Y. Li, X. Yang, Periodic solutions of Fokker-Planck
equations. J. Differential Equations 263 (2017) 285-298.

\bibitem{FengZhaoZhou2011}
C. Feng, H. Zhao, B. Zhou, Pathwise random periodic solutions of
stochastic differential equations. J. Differential Equations 251
(2011) 119-149.

\bibitem{FengWuZhao2016}
C. Feng, Y. Wu, H. Zhao,  Anticipating random periodic solutions-I.
SDEs with multiplicative linear noise. J. Funct. Anal. 271 (2016)
365-417.


\bibitem{Fokam2017}
J. Fokam,  Multiplicity and regularity of large periodic solutions
with rational frequency for a class of semilinear monotone wave
equations. Proc. Amer. Math. Soc. 145 (2017) 4283-4297.

\bibitem{Gao217}
P. Gao, Some periodic type solutions for stochastic
reaction-diffusion equation with cubic nonlinearities. Comput. Math.
Appl. 74 (2017) 2281-2297.

%\bibitem{GasinskiPapageorgiou2018}
%L. Gasi\'{n}ski, N. Papageorgiou, Periodic solutions for nonlinear
%nonmonotone evolution inclusions. Discrete Contin. Dyn. Syst. Ser. B
%23 (2018) 219-238.


\bibitem{IkedaWatanabe1977}
N. Ikeda, S. Watanabe, A comparison theorem for solutions of
stochastic differential equations and its applications. Osaka J.
Math. 14 (1977) 619-633.

\bibitem{JiYi2019}
M. Ji, W. Qi, Z. Shen, Y. Yi, Existence of periodic probability
solutions to Fokker-Planck equations with applications. Journal of
Functional Analysis (in press).


\bibitem{KaratzasShreve1991}
I. Karatzas, S.E. Shreve, Brownian Motion and Stochastic Calculus.
Springer, Berlin, 1991.

\bibitem{KaulVatsala1986}
S.K. Kaul, A.S. Vatsala, Monotone method for integro-differential
equations with periodic boundary conditions. Appl. Anal. 21 (1986)
297-305.

\bibitem{Khasminskii2012}
R.Z. Khasminskii, Stochastic Stability of Differential Equations,
With contributions by G. N. Milstein and M. B. Nevelson, Completely
revised and enlarged second edition, Stochastic Modelling and
Applied Probability, vol. 66, Springer, Heidelberg, 2012.

\bibitem{Krasnoselskii1968}
M. A. Krasnosel'skii, Translations along Trajectories of
Differential Equations. AMS Trans. Math. Monographs 19, 1968.

\bibitem{LakshmtkanthamLeela1984}
V. Lakshmtkantham, S. Leela, Remarks on first and second periodic
boundary value problems. Nonlin. Anal. 8 (1984) 281-287.


\bibitem{LiWangLiLu1999}
Y. Li, H.Z. Wang, X.R. L\"{u}, X.G. Lu, Periodic solutions for
functional-differential equations with infinite lead and delay.
Appl. Math. Comput. 70 (1995) 1-28.

\bibitem{LiCongLinLiu1999}
Y. Li, F.Z. Cong, Z.H. Lin, W.B. Liu, Periodic solutions for
evolution equations. Nonlinear Anal. 36 (1999) 275-293.

\bibitem{LiuSun2014}
Z. Liu, K. Sun, Almost automorphic solutions for stochastic
differential equations driven by L\'{e}vy noise. J. Funct. Anal. 266
(2014) 1115-1149.


\bibitem{LiuWang2016}
Z. Liu, W. Wang, Favard separation method for almost periodic
stochastic differential equations. J. Differential Equations 260
(2016) 8109-8136.

\bibitem{Mellah2013}
O. Mellah, P.R. De Fitte, Counterexamples to mean square almost
periodicity of the solutions of some SDEs with almost periodic
coefficients. Electron. J. Differential Equations 91 (2013) 1-7.

\bibitem{Mohammed1984}
S.E.A. Mohammed, Stochastic Functional Differential Equations.
Researsch Notes in Mathematics. Pitman, Boston, 1984.

\bibitem{PengZhu2006}
S. Peng, X. Zhu, Necessary and sufficient condition for comparison
theorem of 1-dimensional stochastic differential equations. Stoch.
Proc. Appl. 116 (2006) 370-380.

\bibitem{Prato1995}
G.D. Prato, C. Tudor, Periodic and almost periodic solutions for
semilinear stochastic equations. Stochastic Anal. Appl. 13 (1995)
13-33.

\bibitem{SedaNietoCera1992}
V. Seda, J.J. Nieto, M. Cera, Periodic boundary value problems for
nonlinear higher order ordinary differential equations. Appl. Math.
Comput. 48 (1992) 71-82.


\bibitem{Smith1996}
H.L. Smith, Monotone Dynamical Systems. An Introduction to the
Theory of Competitive and Cooperative Systems. Amer Math Soc,
Providence Rhode Island, 1996.

\bibitem{TaralloZhou2018}
M. Tarallo, Z. Zhou, Limit periodic upper and lower solutions in a
generic sense. Discrete Contin. Dyn. Syst. 38 (2018) 293-309.

\bibitem{Tudor1992}
C. Tudor, Almost periodic solutions of affine stochastic evolution
equations. Stochastics Stochastics Rep. 38 (1992) 251-266.







\bibitem{Zhao2009}
H. Zhao, Z. Zheng, Random periodic solutions of random dynamical
systems. J. Differential Equations 246 (2009) 2020-2038.
\end{thebibliography}
\end{document}